\documentclass[12pt]{amsart}
\usepackage{amsmath,amsfonts,amssymb,amsthm,amscd, latexsym, euscript, xypic, verbatim}
\theoremstyle{plain}

\newtheorem*{conj*}{Root Groups Conjecture}
\newtheorem*{thm1.2}{(1.2) Theorem}
\newtheorem*{thm1.3}{(1.3) Theorem}
\newtheorem*{thm1.4}{(1.4) Theorem}
\newtheorem*{prop*}{Proposition}

\newtheorem{prop}{Proposition}[section]

\newtheorem{thm}[prop]{Theorem}

\newtheorem{lemma}[prop]{Lemma}

\theoremstyle{definition}

\newtheorem*{Def*}{Definition}

\newtheorem*{notation*}{Notation}
\newtheorem{remark}[prop]{Remark}

\numberwithin{equation}{section}
\begin{document}


\title{On the third homotopy group of Orr's space}
\author[Emmanuel D.~Farjoun and Roman Mikhailov]
{Emmanuel~D.~Farjoun  and Roman Mikhailov}

\address{Emmanuel D.~Farjoun\\
        Department of Mathematics\\
        Hebrew University of Jerusalem, Givat Ram\\
        Jerusalem 91904\\
        Israel}
\email{farjoun@math.huji.ac.il}
\date{\today}

\address{Roman Mikhailov, Chebyshev Laboratory, St. Petersburg State University, 14th Line, 29b,
Saint Petersburg, 199178 Russia\\ and  St. Petersburg Department of
Steklov Mathematical Institute\\
and School of Mathematics, Tata Institute of Fundamental Research\\
Mumbai 400005, India}
\email{romanvm@mi.ras.ru}
\maketitle

\begin{abstract} K. Orr defined a Milnor-type invariant of links that lies in the third homotopy
group of a certain space $K_\omega.$ The problem of non-triviality of this third homotopy group has been open. We show that it is an infinitely generated group. The question of realization of its elements as links remains open.
\end{abstract}

\section{Introduction}

In [M2] John Milnor defines his $\bar\mu$-invariants of links, which extract numerical invariants from the lower central series quotients of the link group and the homotopy classes of the link longitudes. His questions motivated generations of research concerning these invariants. While most of Milnor’s original questions have been answered (see, for instance, \cite{O1}, \cite{O2}, \cite{IO}, \cite{C1}, \cite{HL}), one key question remains open. Milnor asked if we can extract similar invariants from $\pi/\pi_\omega$. The first candidate for a transfinite Milnor invariant was given by K. Orr in \cite{O1}, where he suggested a possible domain for these invariants, a space he denoted $K_\omega$. A transfinite Milnor invariant is then an element in the third homotopy group, $\pi_3$, of this space. Extensions and refinements were introduced by J. P. Levine in [L1, L2], as mentioned toward the end of the introduction.

The space defined by Orr, denoted by  $K_\omega,$ is the mapping cone  of the natural map
$$
K(F,1)\to K(\widehat F,1).
$$
Here $\widehat F$ is the pro-nilpotent completion of $F$, i.e. $\widehat F={\sf lim}\ F/\gamma_i(F),$ where $\{\gamma_i(F)\}_{i\geq 1}$ is the lower central series of $F$. Let $L\subset S^3$ be a link, $G=\pi_1(S^3\setminus L)$ its group and $f: F\to G$ a meridian homomorphism, where $F$ is the free group of the same rank and number of components of $L$. Assume that all Milnor's invariants  $\bar\mu$-invariants of $L$ are zero. This means in fact that the homomorphism $f$ induces isomorphisms of all finite lower central quotients and therefore, an isomorphism $\widehat F\simeq \widehat G$. Thus we have a map $S^3\setminus L\to K(\widehat F,1)$, which extends to a map $S^3\to K_\omega$. The homotopy class of this map in $\pi_3(K_\omega)$ defines an invariant $\theta_\omega$; see \cite{O2} for details and discussion. The main aim of the present note is to show that the the third homotopy group $\pi_3 K_\omega$ is non trivial, in fact that it is  infinitely generated.

The space $K_\omega$ is simply-connected. This follows from the fact that, for a finitely generated  group $G$, the pro-nilpotent completion $\widehat G$ is normally generated by the images of $G$ under the natural injection $G\to \widehat G$; (see 2.2 below, observe that, this is not true  for a free group of infinite rank) see \cite{Bousfield}. For the case $G=F$ the free group with at least  two generators it is shown in \cite{Bousfield} that $H_2(\widehat F)$ is uncountable, hence the same is true for $\pi_2(K_\omega)\simeq H_2(K_\omega)\simeq H_2(\widehat F).$ Cochran and Orr conjectured that
 this the relevant third homotopy group is non vanishing: The following is written in \cite{C}: ``{\it Bousfield has shown that $\pi_2(K_\omega)$ is infinitely generated and it appears likely that the same will hold for all the homotopy groups of $K_\omega$.}"

 In light of Orr's Milnor-type invariant  in $\pi_3(K_{\omega})$ the question about non-triviality of higher homotopy groups of $K_\omega$ is of interest, and seems to have been open. The following is the main result of the present note.

\begin{thm}\label{mainresult}
The third homotopy group
 $\pi_3(K_\omega)$ is  infinitely generated.
\end{thm}
\noindent  Since we cannot see why  $H_3(K_\omega)$ is not zero, we concentrate attention on  the kernel of Hurewicz map in dimension three in order to carry out  the argument. In fact, we prove that the kernel of the (surjective) Hurewicz homomorphism $\pi_3(K_{\omega}) \to H_3(K_{\omega})$ is an infinitely generated.

Observe that the nontriviality of $\pi_3(K_\omega)$ does not imply the existence of links with non-zero invariants $\theta_\omega$. In order to solve that realization problem, another space $K_\infty$ was defined in \cite{L1}. The definition of $K_\infty$ is similar to $K_\omega$, the difference being  that the algebraic closure is used instead of pro-nilpotent completion. Any element from $\pi_3(K_\infty)$ can be realized as an invariant for a certain link; however, the problem of whether   the Levine group $\pi_3(K_\infty)$ is non-zero is still very much open. The claim, without proof,  in Orr's thesis \cite{O1}
that all elements in his group  $\pi_3K_\omega$  are realizable as links is still unsubstantiated.  In fact, there is a well defined map
relating Levine's group and Orr's:  $l: \pi_3 K_\infty\to \pi_3 K_\omega.$  It is not hard to see from Levine arguments that the elements in the image
of this map are exactly the realizable elements.  We do not know whether this map  $l$ is surjective. In light of out results below it is probably not surjective.

\section{Preliminaries}

\subsection{Whitehead exact sequence}
In order to get a hold on $\pi_3(K_\omega)$ we use an approach due to J.H.C.Whitehead, who defined a certain exact sequence
relating the homology and homotopy groups of a space, and has a transparent form for 1-connected spaces such as $K_\omega.$
Recall the definition of J.\, H.\, C. Whitehead's {\it universal quadratic functor},
$\Gamma^2$, introduced in~\cite[Chapter II]{Wh}. For an abelian
group $A$, $\Gamma^2(A)$ is the group generated by the symbols
$\gamma(x)$, one for each $x\in A$, subject to the defining
relations
\begin{enumerate}
\item $\gamma(-x)=\gamma(x)$; \item
$\gamma(x+y+z)-\gamma(x+y)-\gamma(y+z)-\gamma(x+z)+\gamma(x)+\gamma(y)+\gamma(z)=0.$
\end{enumerate}
It follows from \cite{EM} that this group is naturally isomorphic to the fourth homology group
$$\Gamma^2(A)\cong H_4\,K(A,2).$$

Let us now recall a natural exact sequence associated to a connected pointed space $X$, due to J.H. C. Whitehead.
Let $X$ be a  pointed connected CW-complex with skeletal filtration
$$\mathrm{sk}_1(X)\subset \mathrm{sk}_2(X)\subset \cdots$$
Whitehead constructs the following long exact sequence in~\cite{Wh},
\begin{equation}\label{whitehead}
\cdots \pi_4(X)\buildrel{Hur_4}\over\longrightarrow H_4(X)\longrightarrow
\Gamma_3(X)\longrightarrow
\pi_3(X)\buildrel{Hur_3}\over\longrightarrow H_3(X)\longrightarrow
0,
\end{equation}
where
$$
\Gamma_3(X):=\mathrm{im}(\pi_3(\mathrm{sk}_{2}(X))\to
\pi_3(\mathrm{sk}_3(X)))
$$
and $Hur_i$ is the $i$th Hurewicz homomorphism.

There is a natural map $\Gamma^2(\pi_2(X))\to \Gamma_3(X)$, constructed as follows. Let $\eta\colon S^3\to S^2$ be the Hopf
map and let $x\in \pi_2(X)$ be expressed as the  following composition:
$$
S^2\to \mathrm{sk}_2(X)\hookrightarrow \mathrm{sk}_3(X).
$$
Then the composite
$$
S^3\buildrel{\eta}\over\to S^2\to\mathrm{sk}_2(X)\hookrightarrow
\mathrm{sk}_3(X)
$$
defines an element $\eta^*(x)\in \Gamma_3(X)$. According
to~\cite[Section 13]{Wh}, for a simply connected space $X$ the function:
\begin{equation}\label{eta1}
\eta_1\colon \Gamma^2(\pi_2(X))\to \Gamma_3(X), \ \
\gamma(x)\mapsto \eta^*(x)
\end{equation}
defines a natural homomorphism which  for a simply connected space $X$ is an isomorphism of groups.

In more direct terms, for a simply connected  space  the  following part of  Whitehead's exact sequence
\begin{equation}\label{whitehead1}
\pi_4(X)\to  H_4(X)\longrightarrow
H_4 K(\pi_2(X),2)\longrightarrow
\pi_3(X)\buildrel{Hur_3}\over\longrightarrow H_3(X)\longrightarrow
0,
\end{equation}
is a part of the Leray-Serre spectral sequence of the integral homology  of the (two connected cover-)
fibre sequence $ X[2]\to X\to K(\pi_2(X),2).$ In particular, we have an isomorphism $\Gamma^2 (\pi_2(X))\cong \Gamma_3 (X)\cong H_4 K(\pi_2 (X),2)$ for any  1-connected space $X.$

It follows directly from  Whitehead's exact sequence that in order to estimate $\pi_3(X)$ for a 1-connected space,  we
might  consider the cokernel of the natural map  of Whitehead

$$w_X\equiv w: H_4(X)\to \Gamma^2(\pi_2(X))\cong H_4K(H_2(X),2).$$

The evaluation of this cokernel for $X=K_\omega$  is  thus our main  concern from now on.
\subsection{Completions}
The following lemma follows the argument of Bousfield written in details in \cite{BS}. We give it here for completeness of exposition, in a more general form.
\begin{lemma}\label{lemmabs}
Let $H$ be a group and $\{x_1,\cdots, x_n\}$ a finite set of elements of $H$. Let
$$
H\unrhd N_1\unrhd N_2\unrhd \cdots
$$
be a central series of $H$, such that $[H, N_i]=N_{i+1},\ i=1,2,\cdots$. Suppose that, for every $i\geq 1,$ the quotient
$H/N_i$ is normally generated by $X\cdot N_i=\{x_1N_i,\cdots, x_nN_i\}.$ Consider the inverse limit:
$$
\Gamma:={\sf lim}\ H/N_i
$$
and the associated  natural map $h: H\to \Gamma$. The group $\Gamma$ is normally generated by the elements $h(X)=\{h(x_1),\cdots, h(x_n)\}.$
\end{lemma}

\begin{proof}
The normal generation condition implies that, for every $u\in N_i,\ i\geq 2,$ there exist elements $u_1,\cdots, u_n\in N_{i-1},$ such that
$$
u \equiv [u_1,x_1]\cdots [u_n,x_1]\mod N_{i+1}.
$$
Let $a=(a_1N_1,a_2N_2,\cdots)\in \Gamma,$ where $a_{i+1}\equiv a_i\mod N_i,\ a_i\in H,\ i\geq 1.$ Suppose that $a_1\in N_1$. We claim that there exist elements $g_1,\cdots, g_n\in \Gamma$, such that
\begin{equation}\label{apresentation}
a=[g_1,h(x_1)]\cdots [g_n,h(x_n)].
\end{equation}
We construct  $n$ elements $(g_i)_{1\leq i\leq n}$  in the inverse limit as a sequence of elements in the corresponding quotient groups:
$$
g_i=(y_1^iN_1,y_2^iN_2,\cdots)\in \Gamma,\ 1\leq i\leq n,
$$
with the elements $y_j^i$ constructed by induction as follows. Since $a_1\in N_1,$ $a_2$ also lies in $N_1.$ There exist elements $u_1,\cdots, u_n\in N_1,$ such that
$$
a_2\equiv [u_1,x_1]\cdots [u_n,x_n]\mod N_3.
$$
We set $y_1^i=y_2^i:=u_i$. Then
$$
a_2\equiv [y_2^1, x_1]\cdots [y_2^n, x_n]\mod N_3.
$$
Suppose we have constructed elements $y_1^i,\cdots, y_k^i,\ 1\leq i\leq n$ such that
\begin{align*}
& y_{j+1}^i\equiv y_j^i\mod N_j,\ 1\leq j\leq k-1\\
& a_j\equiv [y_j^1,x_1]\cdots [y_j^n, x_n]\mod N_{j+1},\ 1\leq j\leq k
\end{align*}
and
$$
a_{k+1}\equiv [y_k^1,x_1]\cdots [y_k^n, x_n]\mod N_{k+1}.
$$
There exists $r_{k+2}\in N_{k+1},$ such that
$$
a_{k+1}=[y_k^1,x_1]\cdots [y_k^n, x_n]r_{k+2}.
$$
Now we find elements $v_i\in N_k,$ such that
$$
r_{k+2}\equiv [v_1,x_1]\cdots [v_n,x_n]\mod N_{k+2}.
$$
We set $y_{k+1}^i:=y_k^iv_i,\ 1\leq i\leq n$. It follows that
$$
a_{k+1}\equiv [y_{k+1}^1, x_1]\cdots [y_{k+1}^n, x_n]\mod N_{k+2}.
$$
Now the result follows from the condition that $H/N_1$ is normally generated by elements $x_1N_1,\cdots, x_nN_n$ and the presentation (\ref{apresentation}).
\end{proof}
\begin{remark}\label{normal}
Let $G$ be a finitely generated group and $\{\gamma_i(G)\}_{i\geq 1}$ its lower central series. The completion map $G\to \widehat G:={\sf lim}\ G/\gamma_i(G)$ induces isomorphisms of lower central quotients $G/\gamma_i(G)\buildrel{\sim}\over\to \widehat G/\gamma_i(\widehat G),\ i\geq 1$; (see \cite{BS}). Hence, by Lemma \ref{lemmabs}, $\widehat G$ is normally generated by images of $G$.
\end{remark}
\subsection{The second homology of nilpotent completions}  We need certain information on the second homology groups, to be used in the proof of the main theorem.  In this paragraph we recall a useful relation between the second homology  $H_2(G)$ of a finitely generated group and that of its nilpotent completion $ H_2(\widehat G).$   Let $F$ be as above a finitely generated free group.  Bousfield shows in \cite{Bousfield} that $H_2(\widehat F)$ is uncountable. In fact, $H_2(\widehat F)$ maps onto the exterior square of 2-adic integers. An exposition of Bousfield's method of study;  the homology of pro-nilpotent completions is given in \cite{IM}. Let $G$ be a finitely presented group, $F$ a free group, $F\to G$ an epimorphism, $R=ker\{F\to G\}.$ The quotients
$$
\mathcal B_k(G):=\frac{R\cap
\gamma_k(F)}{[R,\underbrace{F,\cdots, F}_{k-1}]}
$$
are known as {\it Baer invariants} of $G.$  They are naturally isomorphic to the first (non-additive) derived functors of $k$-th lower central quotients
$G/\gamma_k(G)$. They form a tower of groups.

We will need the following property ;(see lemma 6.1 \cite{IM}):

\begin{prop}\label{lim1} Assume that
$$
{\sf lim}^1 \mathcal B_k(G)=0.
$$

 Then the cokernel of the natural map $H_2(\widehat F)\to H_2(\widehat G)$
is isomorphic to a quotient of the homology group $H_2(G)$ as   in the diagram
$$
\xymatrix{ &H_2(G)\ar[d]\ar@{>>}[rd] \\
H_2(\widehat F)\ar[r]&H_2(\widehat G)\ar[r]& coker }
$$
\end{prop}

The  condition holds, for example, for all finitely presented groups $G$ with finite $H_2(G)$.

\section{Proof of Theorem 1.}

We first prove a  statement similar to the claim of the main theorem but regarding  a  simpler,  virtually nilpotent group $G,$
rather than the free group $F$ above.

Given any group $G$ we consider the space $cof_G$ defined as  the cofibre of

$$
K(G,1)\to K(\widehat G,1).
$$

Thus $cof_F$ is just Orr's space $K_\omega.$
In all our cases this cofibre space will be simply connected, since we only consider groups $G$ for which $\widehat G$ is normally generated by the images of $G$. Notice that it is the case for any finitely generated group; (see Remark \ref{normal}).

Let $G$ be given by:
$$
G=\langle a,b\ |\ b^a=b^{-1},\ a^2=1\rangle=\mathbb Z\rtimes C_2,
$$
where $C_2=\langle a\ |\ a^2=1\rangle$ is the two elements  cyclic group acting nontrivially on the integers $\mathbb Z$.

We are interested in the third homotopy group of $cof_G.$ Note that by an easy computation below the  third homology
$H_3(cof_G)$ is infinitely generated, and thus the third homotopy group of this space does not vanish, the space being simply connected. But we need a more precise statement as follows.

\begin{lemma}\label{KG} Let $G$ be the semidirect product  as above. The kernel of the Hurewicz map
 $$\pi_3(K_G)\twoheadrightarrow H_3(K_G)$$
is infinitely generated.
\end{lemma}

$$$$

\begin{proof}

The following argument  uses the naturality of Whitehead's map  $w:H_4(X)\to \Gamma^2(H_2(X))$ with respect to a sequence of maps involving
rational localizations:

$$
cof_G\longrightarrow cof_1 \longleftarrow {\mathbb Q}_\infty K(\mathbb Z_2\otimes \mathbb  Q\rtimes C_2,1)\longrightarrow \mathbb  Q_\infty K({\mathbb Q}^{\oplus n}\rtimes C_2,1)
$$

Spaces in this sequence are discussed below.

We consider  the cokernel of $w$ for the left hand space $cof_G$ by reducing the claims  to an estimate the cokernel of $w$ for the right most space.

In order to prove the lemma, using the above exact sequence of Whitehead for the case of  1-connected spaces
  it is enough to show that the cokernel of the natural map $H_4(cof_G)\to \Gamma^2(\pi_2(cof_G))$
  is infinitely generated. Since the space $cof_G$  is simply connected we have $H_2( cof_G)\cong \pi_2 (cof_G),$
  thus we shall consider the second homology in what follows.

It will turn out that in the above sequence of spaces, the two maps on the left induce isomorphisms on the cokernel of the Whitehead map $w$
while the map on the right induces a surjection onto an infinitely generate group, thus completing the proof.

First notice that the nilpotent completion of $G$ is just a semi-direct product of the 2-adic integers witht the cyclic group $C_2$ namely,
 $\widehat G\cong  \mathbb Z_2\rtimes C_2.$ This follows from a direct
computation of the nilpotent quotients as a semi-direct product.

Consider  the map between cofibre sequences
$$
\xymatrix{K(\mathbb Z\rtimes C_2,1) \ar@{->}[r]\ar@{->}[d] & K(\mathbb Z_2\rtimes C_2,1)\ar@{->}[r]\ar@{->}[d] & cof_G\ar@{->}[d]\\ K(\mathbb Q\rtimes C_2,1) \ar@{->}[r] & K(\mathbb Z_2\otimes \mathbb Q\rtimes C_2,1) \ar@{->}[r] & cof_1}
$$
Both spaces $cof_G$ and  $cof_1$ are simply-connected. The homology groups $H_i(\mathbb Z\rtimes C_2)$ can be computed using  the usual  spectral sequence for the group extension $\mathbb Z
\hookrightarrow \mathbb Z\rtimes C_2\twoheadrightarrow C_2$. The second stage of this spectral sequence stabilizes due to the existence of the splitting map $C_2\hookrightarrow (\mathbb Z\rtimes C_2):$
$$
E_{i,j}^2=\begin{cases} H_i(C_2,\langle b\rangle)=\mathbb Z/2,\ (i,j)=(2k,1),\ k\geq 0\\ H_i(C_2)=\mathbb Z/2,\ (i,j)=(2k+1,0),\ k\geq 0,\\ \mathbb Z,\ (i,j)=(0,0),\\ 0\ \text{in\ other\ cases}\end{cases}
$$
In particular, we get
$$
H_i(\mathbb Z\rtimes C_2)=\begin{cases} \mathbb Z/2\oplus\mathbb Z/2,\ i\ \text{odd}\\ 0,\ i\ \text{even}>0\end{cases}
$$
In the same way we compute the homology groups $H_i(\mathbb Z_2\rtimes C_2).$ In this case, the spectral sequence has the following form
$$
E_{i,j}^2=\begin{cases} \Lambda^{2k}(\mathbb Z_2),\ (i,j)=(0,2k),\ k\geq 1,\\ H_i(C_2,\langle b\rangle)=\mathbb Z/2,\ (i,j)=(2k,1),\ k\geq 0\\ H_i(C_2)=\mathbb Z/2,\ (i,j)=(2k+1,0),\ k\geq 0,\\ \mathbb Z,\ (i,j)=(0,0),\\ 0\ \text{in\ other\ cases}\end{cases}
$$
In particular,
$$
H_i(\mathbb Z_2\rtimes C_2)=\begin{cases} \mathbb Z/2\oplus\mathbb Z/2,\ i\ \text{odd}\\ \Lambda^{2k}(\mathbb Z_2),\ i=2k,\ k\geq 1\end{cases}
$$
The monomorphism $\langle b\rangle=\mathbb Z\hookrightarrow \mathbb Z_2$ induces isomorphisms of homology groups
$$
H_i(C_2, \langle b\rangle)\simeq H_i(C_2, \mathbb Z_2).
$$
Hence, the homology groups of $cof_G$ are the following
$$
H_i(cof_G)=\begin{cases} \Lambda^{2k}(\mathbb Z_2),\ i=2k,\ k\geq 1,\\ 0,\ i=2k+1\end{cases}
$$
Similarly, the homology groups of $cof_1$ are the following
$$
H_i(cof_1)=\begin{cases} \Lambda^{2k}(\mathbb Z_2\otimes \mathbb Q),\ i=2k,\ k\geq 1,\\ 0,\ i=2k+1\end{cases}
$$
Note that all exterior powers $\geq 2$ of $\mathbb Z_2$ are $\mathbb Q$-vector spaces. To see that, we present the exterior power as a natural quotient of the tensor power, after tensoring it with $\mathbb Z/l, l\geq 2,$ one obtains a finite cyclic group, whose image in the exterior power is zero. That is, the exterior powers $\geq 2$ of $\mathbb Z_2$ are divisible. Since these groups are torsion-free, we conclude that they are $\mathbb Q$-vector spaces.
Thus the map $cof_G\to cof_1$ is a homology equivalence. Since $cof_G$ and $cof_1$ are simply connected, $cof_G \to cof_1$ is a homotopy equivalence.

Now we consider the  cokernel of  the Whitehead natural transformation   $H_4\to \Gamma^2H_2 $ for the 1-connected  space $cof_1$  which is just

$$Coker(\Lambda^4(\mathbb Z_2\otimes\mathbb Q)\to \Gamma^2 \Lambda^2(\mathbb Z_2\otimes\mathbb Q) ). $$

We show that this cokernel maps {\em onto} an infinitely generated  group  by comparing it to  the cokernel of the Whitehead map of
another space as follows:

First we notice that the map $   K(\mathbb Z_2\otimes \mathbb Q\rtimes C_2,1)\to cof_1$ induces an isomorphism on rational homology.
Therefore, we consider the Whitehead cokernel for the rational localization of the  relevant range. That is done by comparing it to another  Eilenberg-MacLane space
built using finite dimensional vector space over the rationals $\mathbb Q.$

For any $n$, there is an epimorphism $\mathbb Z_2\otimes \mathbb Q\to \mathbb Q^{\oplus n}$, which induces a  map between groups
$$
(\mathbb Z_2\otimes \mathbb Q)\rtimes C_2 \to \mathbb Q^{\oplus n}\rtimes C_2
$$
where the action of $C_2$ is negation. The group $\mathbb Q^{\oplus n}\rtimes C_2$ is $\mathbb Q$-perfect. We have a  map between $\mathbb Q$-completions
$$
\mathbb Q_\infty K(\mathbb Z_2\otimes \mathbb Q\rtimes C_2,1)\to \mathbb Q_\infty K(\mathbb Q^{\oplus n}\rtimes C_2,1).
$$
Both spaces $K(\mathbb Z_2\otimes \mathbb Q\rtimes C_2,1)$ and $K(\mathbb Q^{\oplus n}\rtimes C_2,1)$ are virtually nilpotent, hence their $\mathbb Q$-completions are the same as  $\mathbb Q$-localizations, in particular, the $\mathbb Q$-homology groups of their $\mathbb Q$-completions are the same as those of the spaces (see proposition 3.4 in \cite{DDK}).
The natural maps between Whitehead sequences
$$
\xymatrix{H_4\mathbb Q_\infty K(\mathbb Z_2\otimes \mathbb Q\rtimes C_2,1)\ar@{->}[r]\ar@{->}[d] & \Gamma^2(H_2\mathbb Q_\infty K(\mathbb Z_2\otimes \mathbb Q\rtimes C_2,1))\ar@{->}[d]\\
H_4\mathbb Q_\infty K(\mathbb Q^{\oplus n}\rtimes C_2,1) \ar@{->}[r] &\Gamma^2(H_2\mathbb Q_\infty K(\mathbb Q^{\oplus n}\rtimes C_2,1))}
$$
are the following
$$
\xymatrix{\Lambda^4(\mathbb Z_2\otimes \mathbb Q)\ar@{->}[r] \ar@{->>}[d]& \Gamma^2(\Lambda^2(\mathbb Z_2\otimes \mathbb Q))\ar@{->>}[d]\\ \Lambda^4(\mathbb Q^{\oplus n}) \ar@{->}[r] &\Gamma^2(\Lambda^2(\mathbb Q^{\oplus n}))}
$$
Now observe that, for $n=2,3$, the group $\Lambda^4(\mathbb Q^{\oplus n})$ is zero, but the group $\Gamma^2(\Lambda^2(\mathbb Q^{\oplus n}))$ is infinitely generated. Therefore, the cokernel of the upper horizontal map in the last diagram must be an infinitely generated group.

We conclude that the cokernel of the map
$$
H_4(cof_G)\to \Gamma^2(\pi_2(cof))
$$
is an infinite divisible group as claimed.
\end{proof}

\subsection{Proof of the main theorem }

We use the  statement in Lemma \ref{KG} about  $cok_G$ for the semi-direct product $G$ as above  to deduce a similar statement about the free
group $F.$
Consider a free group $F$ of rank 2, and an epimorphism $F\to G$  and construct the map
 between cofibres:
$$
\xymatrix{K(F,1) \ar@{->}[d] \ar@{->}[r] & K(\widehat F,1)\ar@{->}[r]\ar@{->}[d] & K_\omega\ar@{->}[d]\\ K(G,1)\ar@{->}[r] & K(\widehat G,1)\ar@{->}[r] & cof_G}
$$
In light of  $H_1(F)\cong H_1(\widehat F)$ above we have an isomorphism $H_2(\widehat F)\simeq H_2(K_\omega).$
The cokernel of the natural map $$H_2(\widehat F)\simeq H_2(K_\omega)\to H_2(cof_G)=\Lambda^2(\mathbb Z_2)$$
is an isomorphism, by  Proposition \ref{lim1}  to a quotient of $H_2(G)=\mathbb Z/2$ (see preliminaries), but $H_2(cof_G)$ is divisible, hence we conclude that $H_2(\widehat F)\to H_2(cof_G)$ is an epimorphism. Hence the middle vertical map in the
natural map between Whitehead exact sequences, is a surjection:
$$
\xymatrix{H_4(\widehat F)\ar@{->}[r]\ar@{->}[d] & \Gamma^2(H_2(\widehat F))\ar@{->>}[r]\ar@{->>}[d] & coker\ar[r]^{1-1}\ar@{->>}[d]^{?}& \pi_3(K_\omega)\ar@{->}[d]\\ H_4(cof_G)\ar@{->}[r] & \Gamma^2(H_2(cof_G))\ar@{->>}[r] & coker\ar[r]^{1-1}& \pi_3(cof_G)}
$$
It is clear now that the arrow with a question mark must be a surjection:
Since the map $H_2(\widehat F)\to H_2(cof_F)=H_2 K_\omega$ is an epimorphism,  the same is true for $\Gamma^2(H_2(\widehat F))\to \Gamma^2(H_2(cof_G))$, and we conclude that the map
$$
H_4(\widehat F)\to \Gamma^2(H_2(\widehat F))
$$
has a cokernel which maps onto an infinite divisible group. Theorem 1.1 now follows from the Whitehead exact sequence for the simply-connected space $K_\omega$:
$$
H_4(\widehat F)\to \Gamma^2(H_2(\widehat F))\to \pi_3(K_\omega)\to H_3(\widehat F)\to 0.
$$

\section{Localizaion of virtually nilpotent spaces} Here we deduce from the above results  that, for a virtually nilpotent group $G$, the H$\mathbb Z$-localization of $K(G,1)$ is not, in general,  a $K(-,1).$ In fact we consider the simplest possible non trivial example of such group namely the semi-direct product $G$
above.
Recall from \cite{DDK} that, for a virtually nilpotent space $X$, the arithmetic square
$$
\xymatrix{X_{\mathbb Z}\ar@{->}[d] \ar@{->}[r] & X_{P}\ar@{->}[d]\\ X_{\mathbb Q}\ar@{->}[r] & X_{P,\mathbb Q}}
$$
is, up to homotopy, a fibre square. Here $P=\oplus \mathbb Z/p$ over all primes and, for a ring $R$, $X_R$ is the $HR$-localization.

Consider $X=K(G,1)$, where $G$ is the group from the previous proof, i.e. $G=\mathbb Z\rtimes C_2$. Clearly, $G$ is virtually nilpotent, it has an infinite cyclic subgroup of index two. Thus the space $X$ is virtually nilpotent and therefore, the arithmetic square for $X$ is a fibre square. Since the group $G$ is $\mathbb Q$-acyclic, $X_{\mathbb Q}$ is contractible and the arithmetic square degenerates to the following fibre sequence
$$
X_{\mathbb Z}\to X_{\mathbb Z/2}\to X_{\mathbb Z/2,\mathbb Q.}
$$
Clearly we can ignore all primes from $P$ except for $p=2$.

The $\mathbb Z/2$-localization $X_{\mathbb Z/2}$ coincides with $\mathbb Z/2$-completion and is equivalent to\\ $K(\mathbb Z_2\rtimes C_2,1)$. The space $K(\mathbb Z_2\rtimes C_2,1)$ is virtually nilpotent, hence its $\mathbb Q$-localization coincides with its $\mathbb Q$-completion. Therefore
$$
\pi_i(X_{\mathbb Z})\simeq \pi_{i+1}(\mathbb Q_{\infty}K(\mathbb Z_2\rtimes C_2,1)),
$$
for all $i\geq 2$. The natural map $\mathbb Z_2\rtimes C_2\to \mathbb Z_2\otimes \mathbb Q\rtimes C_2$ induces an isomorphism of $\mathbb Q$-homology groups, hence
$$
\pi_i(X_{\mathbb Z})\simeq \pi_{i+1}(\mathbb Q_{\infty}K(\mathbb Z_2\otimes \mathbb Q\rtimes C_2,1)).
$$
It follows from the previous section that the kernel of the third Hurewicz homomorphism for the space $\mathbb Q_{\infty}K(\mathbb Z_2\otimes \mathbb Q\rtimes C_2,1)$ contains an infinitely generated subgroup, hence $\pi_2(X_{\mathbb Z})\neq 0$.

\section*{Acknowledgement}
The research is supported by the Russian Science Foundation grant
N 16-11-10073. The first named author is grateful for the hospitality of Chebyshev Laboratory, St. Petersburg.

\end{document}